\documentclass{article}
\usepackage[round]{natbib}
\usepackage{amsmath,amsfonts,amssymb,amsthm}
\usepackage{paralist}

\renewcommand{\le}{\leqslant}
\renewcommand{\ge}{\geqslant}

\newtheorem{theorem}{Theorem}
\newtheorem{lemma}{Lemma}
\newtheorem{proposition}{Proposition}

\newcommand{\rd}{\,\mathrm{d}}

\newcommand{\real}{\mathbb{R}}

\newcommand{\e}{\mathbb{E}}
\newcommand{\var}{\mathrm{Var}}
\newcommand{\cov}{\mathrm{Cov}}
\newcommand{\wh}{\widehat}

\newcommand{\bsx}{\boldsymbol{x}}

\newcommand{\simiid}{\stackrel{\mathrm{iid}}\sim}

\newcommand{\giv}{\!\mid\!}

\author{Art B. Owen \\ Stanford University
\and
Yi Zhou
}
\title{The square root rule for adaptive importance sampling}
\date{March 2019}

\begin{document}
\maketitle
\begin{abstract}
In adaptive importance sampling, and other contexts, we have $K>1$
unbiased and uncorrelated estimates $\hat\mu_k$ of a common quantity $\mu$. 
The optimal unbiased linear combination weights them inversely
to their variances but those weights are unknown and hard to estimate.
A simple deterministic square root rule based on a working model that $\var(\hat\mu_k)\propto k^{-1/2}$
gives an unbisaed estimate of $\mu$ that is nearly optimal under
a wide range of alternative variance patterns.
We show that if $\var(\hat\mu_k)\propto k^{-y}$
for an unknown rate parameter $y\in [0,1]$ then
the square root rule yields the optimal variance rate
with a constant that is too large by at most $9/8$ for any $0\le y\le 1$ and any number
$K$ of estimates.  Numerical work shows that rule is similarly robust to some other
patterns with mildly decreasing variance as $k$ increases.
\end{abstract}

\section{Introduction}

A useful rule for weighting uncorrelated estimates appears in
an unpublished technical report \citep{mixbeta}.
This paper states and proves that result and discusses some
mild generalizations.
The motivating context for the problem is adaptive importance sampling.
	
Importance sampling is a fundamental Monte Carlo method
used in finance \citep{glas:2004},
reliability \citep{au:beck:1999,aloe}, coding theory \citep{elvi:sant:2019},
particle transport \citep{kong:span:2011,
koll:bagg:cox:pica:1999}, computer graphics \citep{veac:guib:1995,phar:jako:hump:2016},
queuing \citep{blan:glyn:liu:2007}, and sequential analysis \citep{lai:sieg:1977},
among other areas.

The prototypical problem is to estimate
a constant $\mu =\int f(\bsx) p(\bsx)\rd\bsx$
for a density function $p$ and integrand $f$. Although we illustrate the
problem with continuous distributions, discrete distributions are handled similarly.
In importance sampling, we might estimate $\mu$ by 
\begin{align}\label{eq:hatmu}
\hat\mu = \frac1n\sum_{i=1}^n\frac{f(\bsx_i)p(\bsx_i)}{q(\bsx_i;\theta)},\quad  
\text{for $\bsx_i\simiid q(\cdot;\theta)$}. 
\end{align} 
The density function $q$ must satisfy
$q(\bsx;\theta)>0$ whenever $f(\bsx)p(\bsx)\ne 0$,
in order to ensure that $\e(\hat\mu)=\mu$.
In this setting, we must also choose a parameter $\theta$.

A common alternative method is self-normalized importance 
sampling with
\begin{align}\label{eq:tildemu}
\tilde\mu = 
\sum_{i=1}^n\frac{f(\bsx_i)p(\bsx_i)}{q(\bsx_i;\theta)}
\Bigm/\sum_{i=1}^n\frac{p(\bsx_i)}{q(\bsx_i;\theta)}
,\quad 
\text{for $\bsx_i\simiid q(\cdot;\theta)$}. 
\end{align}
This method is more restrictive in requiring $q(\bsx;\theta)>0$ whenever
$p(\bsx)>0$ but less restrictive in 
that it allows $p$ or $q$ or
both to be  unnormalized, i.e., known only up to a constant
of proportionality. We focus primarily on $\hat\mu$, with
a few remarks about $\tilde\mu$.

After choosing $\theta$ and getting $\bsx_1,\dots,\bsx_n$ we might well decide
that some other value of $\theta$ would have been better.  For instance we might
think that some other $\theta$ would yield an estimate of $\mu$ with lower variance.
In adaptive importance sampling
we estimate $\mu$ by $\hat\mu_k$ at iteration $k\ge1$ using some parameter $\theta_k$.
Then we may use all the sample values from the first $k$ iterations to select $\theta_{k+1}$
for the next round.

In an adaptive method with one pilot estimate and one final estimate, $K=2$
and $n$ is usually large.
In other settings, each $\hat\mu_k$ could be based on just one
evaluation of the integrand 
(e.g., \cite{ryu:boyd:2014})
and then $K$ would typically be very large.
In intermediate settings we might have something like $K=10$ estimates
based on perhaps thousands of evaluations each.
For instance, the cross-entropy method \citep{debo:kroe:mann:rubi:2005}
might be used this way.
The total sample size used is $N=nK$. We remark briefly on the
possibility that the value of $n$ varies with $k$ in 
Section~\ref{sec:gene}.

To specify an adaptive importance sampling method, one must come up with
a family $q(\cdot;\theta)$ of distributions, a starting value $\theta_1$, 
a choice for $n$,
an algorithm for computing $\theta_{k+1}$ from
the data of the prior $k$ rounds and a termination rule.
This paper addresses a different issue: after all those choices have
been made, how should one pool estimates
$\hat\mu_1,\dots,\hat\mu_K$ to get the combined estimate $\hat\mu$?
We have two  main points to make:
\begin{compactenum}[\quad\bf1)]
\item a natural approach using inverse sample variances to weight $\hat\mu_k$ is risky, and
\item a simple deterministic weighting scheme is nearly optimal over a broad range of 
reasonable conditions.
\end{compactenum}
We will make the first point with some qualitative arguments and citations.

Most of the paper
is devoted to the second point showing that a simple
deterministic rule is nearly optimal
over a set of assumptions ranging from what we consider unduly pessimistic to what 
we consider unduly optimistic.
To that end we adopt a working model in which $\var(\hat\mu_k)\propto k^{-1/2}$.
Using this working model we get an estimate $\hat\mu$ satisfying $\e(\hat\mu)=\mu$
along with an unbiased estimate of $\var(\hat\mu)$.
The working model will ordinarily be incorrect and the consequence
of that incorrectness is a larger value of $\var(\hat\mu)$ than would have been attained
using the unknown true variances.  We give conditions under which the
resulting inefficiency is mild enough to be negligible.  That lets one avoid
a potentially severe challenge of estimating the step by step variances
and then analyzing a complicated rule derived from those estimates.

This paper is organized as follows. 
Section~\ref{sec:plausible} gives reasons to study settings where
$\var(\hat\mu_k)$ will decay steadily but
not dramatically, and with diminishing returns as $k$ increases.
Section~\ref{sec:nota} presents our notation and 
explains why we can assume that the estimates $\hat\mu_k$
are unbiased and uncorrelated.
Section~\ref{sec:linear} presents our linear combination 
estimator, discusses why it is safer to use a prespecified 
linear combination of $\hat\mu_k$ than to use estimated
variances, and gives an expression 
for the relative inefficiency that stems from using a weighting
derived from a potentially incorrect power law.
Section~\ref{sec:squa} states and proves our main result using 
two lemmas. It is that the square root rule never 
raises the variance by more than $9/8$ times the optimal 
variance for any setting where $\var(\hat\mu_k)\propto k^{-y}$
for any $0\le y\le 1$ and any number $K\ge1$ of steps. 
It is accurate to within that constant factor even though
the unknown optimal $\var(\hat\mu)$ converges at rates
ranging from $O(N^{-1})$ to $O(N^{-2})$.
Section~\ref{sec:robust} considers some alternative ways to model 
steady improvement with diminishing returns.  We find numerically 
that weighting proportionally to $k^{1/2}$ loses little accuracy 
under very general circumstances, although some variance 
patterns give worse than a $9/8$ variance ratio. 
Section~\ref{sec:gene} includes some remarks on how to
weight the $\hat\mu_k$ in settings where even better convergence rates, including 
exponential convergence, apply. 
Section~\ref{sec:conc} has brief conclusions.

\section{Plausible convergence rates}\label{sec:plausible}

The key assumption we need to make is about how quickly the statistical efficiency
of the individual estimates $\hat\mu_k$ increases with $k$.
Very commonly $f(\bsx)\ge0$ and then $f(\bsx)p(\bsx)/\mu$ is a probability
density function.  If we could sample from that density,
then we would have a zero variance estimate $\hat\mu$.
In special circumstances, $f(\bsx)p(\bsx)/\mu=q(\bsx;\theta)$
for some finite dimensional parameter vector $\theta$
in our family of densities.
It may then be possible to get good estimates of 
$\theta_k$ that improve rapidly as $k$ increases,  with the effect that $\var(\hat\mu_k)$ decreases 
exponentially fast to $0$ as $k$ increases.
This situation is very well studied in the literature. See 
\cite{koll:bagg:cox:pica:1999} and \cite{kong:span:2011}
and \cite{lecu:tuff:2008}.
We consider it to be especially favorable and not representative
of importance sampling in general.

In other settings $f(\bsx)p(\bsx)/\mu$ does not take the form $q(\bsx;\theta)$
for any $\theta$ and then $\hat\mu_k$ cannot have zero variance for finite $n$.
It remains possible that $\theta_k$ approaches an optimal vector
and then $\var(\hat\theta_k)$ can improve towards a positive lower bound
but no further.
Exponential convergence is ruled out a fortiori.

The self-normalized estimator $\tilde\mu_k$ cannot
approach zero variance for fixed $n$ outside of trivial settings where $f(\bsx)$
is constant under $\bsx\sim p$.
It is a ratio estimator and the asymptotically optimal 
$q$ is proportional to $|f(\bsx)-\mu|p(\bsx)$.
Any self-normalized estimator satisfies
$$
\lim_{n\to\infty} n\var(\tilde\mu) \ge \biggl(\int |f(\bsx)-\mu|p(\bsx)\rd\bsx\biggr)^2. 
$$
See the Appendix of \cite{aloe} which uses 
the optimal self-normalized importance sampler 
found by \cite{kahn:mars:1953}, also given in \cite{hest:1988}. 
When $f(\bsx)$ is the indicator of an event with probability $\epsilon\in(0,1)$
under $\bsx\sim p$, then the asymptotically optimal self-normalized
sampler has $\Pr(f(\bsx)=1;\bsx\sim q)=1/2$, and any self-normalized importance
sampler has
$$
\lim_{n\to\infty} n\var(\tilde\mu) \ge 4\epsilon^2(1-\epsilon)^2.
$$
As a consequence, any nontrivial instance of the self-normalized importance
sampler must also have a nonzero floor on its mean squared error for fixed $n$.
This rules out exponential convergence for
self-normalized importance sampling in non-trivial problems.

In a third possibility, $q(\bsx;\theta)$ is an infinite dimensional
family of densities in which $p(\bsx)f(\bsx)$ can be approached
to any desired level of accuracy by some choice of~$\theta$.
Following up on a strategy of \cite{west:1993},
the nonparametric importance sampler
of~\cite{zhan:1996} has $K=2$ steps.
The second step is an importance sample from a kernel density estimate
of $p(\bsx)f(\bsx)/\mu$ obtained using the first step as a pilot sample.
He gives conditions under which the mean squared error at the 
second stage is $O(N^{-(d+8)/(d+4)})$ for $\bsx\in\real^d$.
A multistep adaptive algorithm attains the same rate in $N$ with a
different constant that is not necessarily better than using $K=2$.

We are targetting problems where exponential convergence
is not possible. We assume that
$\var(\hat\mu_k)$ reduces steadily with $k$ but not spectacularly so
and most of our analysis assumes diminishing returns as $k$ keeps increasing.
Our main model for steady variance reduction is $\var(\hat\mu_k)\propto k^{-y}$
for an unknown value of $y$.
While the actual variance won't follow this pattern exactly we expect
that this pattern is qualitatively similar to realistic progress and
it is amenable to sharp analysis.
Our combined estimate is a linear combination of $\hat\mu_k$.
The resulting variance is a linear combination of $\var(\hat\mu_k)$
and so mild departures from a $k^{-y}$ pattern will have little
effect on the conclusion.

We consider the case with $y=0$ to be unduly pessimistic.  It describes
a setting where adaptation brings no benefit whatsoever.
We consider $y=1$ to be unduly optimistic. While it is not as favorable
as the exponential case described above, we show below that it yields
an asymptotic mean squared error of $O(N^{-2})$. That is better
than the nonparametric method can attain in even one dimensional
problems and it is much better than we can attain when $p(\bsx)f(\bsx)/\mu$
cannot be approximated arbitrarily well with some $q(\bsx;\theta)$.

When $f(\bsx)p(\bsx)/\mu$ cannot be arbitrarily well approximated by $q(\bsx;\theta)$ for any $\theta$
then it is not actually possible to have $\var(\hat\mu_k)\propto k^{-y}$ for any $y>0$ and all $k\ge1$.
This exponential decay remains a reasonable working model for initial transients over $1\le k\le K$
for some finite $K$.

Our proposed estimate is to weight the interates $\hat\mu_k$
as if $y=1/2$.  The result is $\hat\mu = \sum_{k=1}^Kk^{1/2}\hat\mu_k
\sum_{k=1}^Ki^{1/2}$, a square root rule.

\section{Notation}\label{sec:nota}

Step $k$ of our adaptive importance sampler
generates data $X_k$. This includes all of the sample points $\bsx_{ki}$
along with the function evaluations $f(\bsx_{ki})$, $p(\bsx_{ki})$ and $q(\bsx_{ki};\theta_k)$
at that step.  It can also include $q(\bsx_{ki};\theta')$ for values $\theta'\ne \theta_k$.
We let $Z_k = (X_1,\dots,X_{k-1})$
denote the data from all steps prior to the $k$'th, with $Z_1$ being empty.
We assume that our estimates $\hat\mu_k$ satisfy
\begin{align}\label{eq:unbiased}
\e(\hat\mu_k\giv Z_k)=\mu,\quad
\text{and}\quad \var(\hat\mu_k\giv Z_k)=\sigma_k^2<\infty,\quad
 k=1,\dots,K.
\end{align}
This is a reasonable assumption for the estimates in equation~\eqref{eq:hatmu},
though we are ruling out extreme settings with $\sigma^2_k=\infty$.
The self-normalized estimate $\tilde\mu$ of~\eqref{eq:tildemu} does not satisfy~\eqref{eq:unbiased}
because as a ratio estimator it incurs a bias of size $O(1/n)$. For large $n$, the bias
might be small enough that  equation~\eqref{eq:unbiased}
becomes a good  approximation for the self-normalized case also.

For the regular importance sampler in equation~\eqref{eq:hatmu},
\begin{align}\label{eq:varmuhatk}
\var(\hat\mu_k) = \e(\var(\hat\mu_k\giv Z_k)) +
\var(\e(\hat\mu_k\giv Z_k)) = \e(\sigma^2_k)\equiv\tau^2_k.
\end{align}
There is ordinarily an unbiased estimate of $\sigma^2_k$, such as
$$
\hat\sigma^2_k = \frac1{n-1}\sum_{i=1}^n\Bigl(\frac{ f(\bsx_i)p(\bsx_i)}{q(\bsx_i;\theta_k)}-\hat\mu_k\Bigr)^2.
$$
Under some moment conditions, $\e(\hat\sigma^2_k\giv Z_k)=\sigma^2_k$ and then
$$
\e(\hat\sigma^2_k) = \e(\sigma_k^2) = \tau^2_k.
$$
That is, $\hat\sigma^2_k$ is simultaneously an unbiased estimate of both
the random variable $\var(\hat\mu_k\giv Z_k)$ and the constant $\var(\hat\mu_k)$.

In \cite{mixbeta}, $\hat\mu_k$ was an importance sampled
estimate of an integral over the unit cube, sampling from a mixture
of products of beta distributions. Then $\theta_k$ included the 
mixture components and beta distribution parameters used at step $k$.
That paper considers a synthetic integrand formed by a mixture of Gaussians
and an integrand with positive and negative singularities from particle physics.

The estimates $\hat\mu_k$ are, in general, dependent because data sampled at step
$k$ influences the choice of $\theta_\ell$ for $\ell>k$ and thereby affects $\hat\mu_\ell$.
These estimates are uncorrelated because
\begin{align*}
\cov(\hat\mu_k,\hat\mu_\ell)  & = \e( \e( (\hat\mu_k-\mu) (\hat\mu_\ell-\mu)\giv Z_{\ell})) \\
&= \e( (\hat\mu_k-\mu) \e(\hat\mu_\ell-\mu\giv Z_{\ell}))=0. 
\end{align*}
The underlying idea in the above expressions is that
$\sum_{k=1}^L\omega_k(\hat\mu_k-\mu)$ is a martingale in $L$
\citep{will:1991}. We will not make formal use of martingale arguments.

\section{Linear combinations}\label{sec:linear}

A simple strategy is to estimate $\mu$ by 
$$\hat\mu=\sum_{k=1}^K\omega_k\hat\mu_k,\quad\text{where}\quad
\sum_{k=1}^K\omega_k=1.$$
For deterministic $\omega_k$, we have $\e(\hat\mu) = \mu$
with $\var(\hat\mu) = \sum_{k=1}^K\omega_k^2\tau_k^2$
and then $\wh\var(\hat\mu)=\sum_{k=1}^K\omega_k^2\hat\sigma^2_k$
is an unbiased estimate of $\var(\hat\mu)$ even if $\var(\hat\mu_k)$
is not proportional to $\omega_k^{-1}$.

To minimize $\var(\hat\mu)$, we should take $\omega_k\propto \var(\hat\mu_k)^{-1}$. 
The problem is that we do not know $\var(\hat\mu_k)$. 
It is not advisable to plug in the unbiased estimates $\hat\sigma^2_k$
and take $\omega_k\propto\hat\sigma^{-2}_k$, for several reasons that we outline next.

The importance sampling context is often one where $f(\bsx)$ is
the indicator of a very rare event.  
For background on importance sampling for rare events, see~\cite{lecu:mand:tuff:2009}. 
Another common setting is for $f(\bsx)$
to be a nonnegative random variable with a very heavy right tail, perhaps
even an integrable singularity. Both of these cases describe a setting
where $f(\bsx)$ has extremely large positive skewness under $\bsx\sim p$.
We then expect to get $f(\bsx)p(\bsx)/q(\bsx)$ with large positive skewness
under $\bsx\sim q$.  \cite{evan:swar:1995} describe this as low effective
sample size for the integrand $f$.
The more general low effective sample size problem
shows up as extreme positive skewness in $p(\bsx)/q(\bsx)$
under $\bsx\sim q$.

In cases of positive skewness, we get a positive correlation between
the sample values of $\hat\mu_k$ and $\hat\sigma^2_k$ \citep{mill:1986}.
Weighting points proportionally to $\hat\sigma^{-2}_k$ then means
that the largest $\hat\mu_k$ tend to be randomly downweighted while
the smallest ones tend to be randomly upweighted.  The result is a negative bias
in $\hat\mu$, which is especially undesirable when $\mu$ is the probability of a rare adverse event.
For a rare event that fails to happen even once at step $k$ we could
get $\hat\mu_k=\hat\sigma^2_k=0$ and then some complicated
intervention would be required to make the formulas work.
It is possible that some settings will have negative skewness for $fp/q$ and a 
corresponding upward bias using estimated variances, though that is less to be expected in importance sampling.
This could happen for problems with negative integrable singularities or problems
with both positive and negative integrable singularities or when estimating the
probability of an event whose complement is rare.

In addition to the bias issue there is a possibly more fundamental reason 
not to use $\hat\sigma^2_k$ in weights.  Importance sampling is often 
used in settings where simply estimating $\mu$ is quite hard. 
Estimating $\sigma^2_k$ or $\tau^2_k$ well could then be an order 
of magnitude harder.  They only have finite variances under a finite 
fourth moment condition for $fp/q$.  Using estimates of those 
quantities then amounts to using an extremely noisy weighting 
relying on a subsidiary problem (variance estimation) that is 
harder than the motivating problem.  
See~\cite{lecu:blan:tuff:glyn:2010} for an analysis of how 
hard it is to estimate variance in the importance sampling context. 
This issue is especially 
severe in cases with large $K$ and small~$n$. 
\cite{chat:diac:2018} also remark on the difficulty of using 
variances in importance sampling.

The AMIS algorithm of 
\cite{corn:mari:mira:robe:2012} uses a sample determined weighted combination of estimates.
In that case the weight applied to $\hat\mu_k$ can depend 
on observations from future rounds in a more complicated way 
than just normalizing $\hat\sigma^{-2}_k$.  
The intricate dependence pattern between weights $\hat\omega_k$ and estimates $\hat\mu_k$
complicates even the task of proving consistency of the estimator.
Their setup was for self-normalized importance sampling but 
the difficulties with noisy estimates of optimal weights apply there too.
Even without that complexity taking account of sampling fluctations
in both $\hat\mu_k$ and $\hat\omega_k$ can raise challenging issues
that a deterministic rule avoids.

We are interested in cases where adaptation brings steady
but not dramatic improvements.
As noted above, our model for that is
\begin{align}\label{eq:truevar}
\var(\hat\mu_k) = \tau^2 k^{-y}
\end{align}
for $0\le y\le 1$ and $\tau=\tau_1\in(0,\infty)$. 
If we worked with $\tau^2 k^{-x}$ then our estimate would be
$$
\hat\mu =\hat\mu(x) = \sum_{i=1}^Ki^{x}\hat\mu_i\Bigm/\sum_{i=1}^Ki^{x}.
$$
The best choice is $\hat\mu(y)$ but $y$ is unknown. Our variance using $x$ is
$$
\var(\hat\mu(x)) = \tau^2\sum_{i=1}^Ki^{2x-y}\Bigm/\Biggl(\,\sum_{i=1}^Ki^{x}\Biggr)^2 
$$
and we measure the inefficiency of our choice by
\begin{align}\label{eq:ineff}
\rho_K(x\giv y) \equiv \frac{\var(\hat\mu(x))}{\var(\hat\mu(y))}
=
\frac{\bigl(\sum_{i=1}^Ki^{2x-y}\bigr)\bigl(\sum_{i=1}^Ki^{y}\bigr)}
{\bigl(\sum_{i=1}^Ki^{x}\bigr)^2}. 
\end{align}

If the variance of $\hat\mu_k$ decays as $k^{-y}$ and, knowing that, we use
$x=y$, then $\var(\hat\mu) = O(K^{-y-1})=O(N^{-y-1})$.  
The value $y=0$ is pessimistic as it corresponds to a setting where adaptation brings no
benefits.
For an optimistic value $y=1$, the variance decays at the
rate $O(N^{-2})$ which we consider unreasonably optimistic.
It is better than the rate that \cite{zhan:1996} gets for importance
sampling by one dimensional kernel density estimates
and even slightly better than the rate
which holds for randomly shifted lattice rules applied to
functions of bounded variation in the sense of Hardy and Krause.
See \cite{lecu:lemi:2000} for background on randomized lattice rules.

\section{Square root rule}\label{sec:squa}

\cite{mixbeta} proposed to split the difference between
 the optimistic estimate $\hat\mu(1)$ and the pessimistic
one $\hat\mu(0)$ by using $\hat\mu(1/2)$.
The result is an unbiased estimate that attains the
same convergence rate as the unknown best estimate
with only a modest penalty in the constant factor.

\begin{theorem}\label{thm:sqrtrule}
For $\rho_K(x\giv y)$ given by equation~\eqref{eq:ineff},
\begin{align}\label{eq:ninebyeight}
\sup_{1\le K<\infty}\,\sup_{0\le y\le 1}\rho_K\Bigl(\frac12\bigm| y\Bigr)\le \frac98.
\end{align}
If $x\ne 1/2$ and $K\ge2$ then
\begin{align}\label{eq:halfisminimax}
\sup_{0\le y\le1}\rho_K(x\giv y)>\sup_{0\le y\le1}\rho_K\Bigl(\frac12\bigm| y\Bigr).
\end{align}
\end{theorem}

\smallskip
Equation~\eqref{eq:ninebyeight} shows that the square root rule
has at most $12.5\%$ more variance than the unknown best rule.
Equation~\eqref{eq:halfisminimax} shows that the choice $x=1/2$
is the unique minimax optimal one, when $K>1$.
When $K=1$ then $\hat\mu(x)$ does not depend on $x$
and in that case, $\hat\mu(x)=\hat\mu(y)$.
Before proving the theorem, we establish two lemmas.

\begin{lemma}\label{lem:supovery}
For $\rho_K(x\giv y)$ given by equation~\eqref{eq:ineff},
with $K\ge1$ and $0\le x\le1$, 
$$\sup_{0\le y\le1} \rho_K(x\giv y)=
\begin{cases}
\rho_K(x\giv 1), & x\le1/2\\
\rho_K(x\giv 0), & x\ge1/2. 
\end{cases}
$$
\end{lemma}
\begin{proof}
The result holds trivially for $K=1$ because $\rho_1(x\giv y)=1$.
Now suppose that $K\ge2$. Then
$$
\frac{\partial^2}{\partial y^2}\rho_K(x\giv y)
= C^{-2}\times \sum_{i=1}^K\sum_{j=1}^Ki^{2x-y}j^y(\log(i)-\log(j))^2
$$
for $C=\sum_{i=1}^Ki^x$. Thus $\rho_K(x\giv y)$ is strictly
convex in $y$ for any $x$ when $K\ge2$, and so
$\sup_{0\le y\le 1}\rho_K(x\giv y)\in \{\rho_K(x\giv1),\rho_K(x\giv0)\}$.

By symmetry, $\rho_K(x\giv 2x-y)=\rho_K(x\giv y)$.
So if $x\ge1/2$, then $\rho_K(x\giv 0)=\rho_K(x\giv 2x)\ge\rho_K(x,1)$.
The last step follows because $2x\ge1$ and $\rho_K(x\giv y)$ is a convex
function of $y$ with its minimum at $y=x\le1$.
This establishes the result for $x\ge1/2$ and a similar argument
holds for $x\le 1/2$. For $x=1/2$, $\rho_K(x\giv0)=\rho_K(x\giv1)$.
\end{proof}

Proposition~\ref{prop:ibounds} has two integral bounds that we will use below.
\begin{proposition}\label{prop:ibounds}
If $0\le x\le1$, and $K\ge1$ is an integer, then
\begin{align}
\sum_{i=1}^{K} i^x&\ge \int_{1/2}^{K+1/2} v^x\rd v = \frac{(K+1/2)^{x+1}-(1/2)^{x+1}}{x+1},\quad\text{and}
\label{eq:inthigh}\\
\sum_{i=1}^K i^x &\le \int_1^{K+1} v^x\rd v = \frac{(K+1)^{x+1}-1}{x+1}. 
\label{eq:intlow}
\end{align}
Equation~\eqref{eq:inthigh} is strict when $0<x<1$.
Equation~\eqref{eq:intlow} is strict when $0<x\le1$.
\end{proposition}
\begin{proof}
Equation~\eqref{eq:inthigh} follows from the concavity of $v^x$.
That concavity is strict for $0<x<1$.
Equation~\eqref{eq:intlow} follows because $v^x\ge i^x$ for
$i\le v\le i+1$ which holds strictly for $x>0$.
\end{proof}
Equation~\eqref{eq:inthigh} is much
sharper than the bound one gets by integrating $v^x$ over $0\le v\le K$.


\begin{lemma}\label{lem:monok}
For $\rho_K(x\giv y)$ given by equation~\eqref{eq:ineff},
$\rho_{K+1}(1/2\giv1)>\rho_K(1/2\giv1)$
holds for any integer $K\ge1$. 
\end{lemma}
\par\noindent
\begin{proof}
Let $S_K=\sum_{i=1}^Ki^{1/2}$. Then  using~\eqref{eq:inthigh} from Proposition~\ref{prop:ibounds},
\begin{align*}
\frac{\rho_{K+1}(1/2\giv1)}{\rho_K(1/2\giv1)}
&=\frac{ (K+1)(K+2)S_K^2}{K^2(S_{K}+\sqrt{K+1})^2}
=\frac{ (K+1)(K+2)}{K^2\bigl(1+\frac{\sqrt{K+1}}{S_K}\bigr)^2}\notag\\
&>\frac{ (K+1)(K+2)}{K^2}
\Bigm/
\Bigl(1+\frac{\sqrt{K+1}}{ ((K+1/2)^{3/2}-2^{-3/2})/(3/2)}\Bigr)^{2}\notag\\
&=(K+1)(K+2)\Bigm/\Bigl(K+\frac32\frac{K}{ f(K) -2^{-3/2}/\sqrt{K+1}}\Bigr)^{2},
\end{align*} 
for
$$
f(K)=(K+1/2)^{3/2}/\sqrt{K+1}.
$$

Let $g(K)=f(K)/(K+1/4)$. We easily find that $\lim_{K\to\infty}g(K)=1$.
Also $g$ is monotone decreasing in $K$ because the derivative of $\log(g(K))$
with respect to $K$ is $-(3/16)[(K+1/2)(K+1/4)(K+1)]^{-1}$.
Therefore, if $K>7$,
\begin{align}\label{eq:thatratio}
\frac{\rho_{K+1}(1/2\giv1)}{\rho_K(1/2\giv1)}
&>(K+1)(K+2)\Bigl(K+\frac32\frac{K}{ g(K)(K+1/4)-2^{-3/2}/\sqrt{8}}\Bigr)^{-2}\notag\\
&>(K+1)(K+2)\Bigl(K+\frac32\frac{K}{ K+1/4-2^{-3/2}/\sqrt{8}}\Bigr)^{-2}\notag\\
&=\frac{ (K+1)(K+2)}{\bigl(K+\frac32\frac{K}{ K+1/8}\bigr)^2}.
\end{align}
The numerator in~\eqref{eq:thatratio} is $K^2+3K+2$
while the denominator is
$$K^2+3K\frac{K}{K+1/8}+\frac94\frac{K^2}{(K+1/8)^2}.$$
The numerator is larger than the denominator, establishing the theorem
for $K>7$.  For $1\le K\le 7$
we compute  directly that 
$$\frac{\rho_{K+1}(1/2\giv1)}{\rho_K(1/2\giv1)}>1.0038.\qedhere$$
\end{proof}
\smallskip

\begin{proof}[Proof of Theorem~\ref{thm:sqrtrule}]
Applying Lemmas~\ref{lem:supovery} and~\ref{lem:monok},
\begin{align*}
\sup_{1\le K<\infty}\,\sup_{0\le y\le 1}\rho_K\Bigl(\frac12\bigm| y\Bigr)
=\lim\limits_{K\to\infty}\rho_K\Bigl(\frac12\bigm| 1\Bigr)
=\lim_{K\to\infty}\frac{K^2(K+1)/2}
{\bigl(\sum_{i=1}^Ki^{1/2}\bigr)^2}. 
\end{align*}
Using the integral bounds~\eqref{eq:inthigh} and~\eqref{eq:intlow} from Proposition~\ref{prop:ibounds}
we can show that the above limit is $9/8$, establishing equation~\eqref{eq:ninebyeight}.
Next, if $x<1/2$, then 
\begin{align*}
\sup_{0\le y\le 1}\rho_K(x\giv y)=\rho_K(x\giv1)
&=\frac{\bigl(\sum_{i=1}^Ki^{2x-1}\bigr)\bigl(\sum_{i=1}^Ki\bigr)}
{\bigl(\sum_{i=1}^Ki^{x}\bigr)^2}\\
&=
\frac{K(K+1)}2 
\frac{\bigl(\sum_{i=1}^Ki^{2x-1}\bigr)}{\bigl(\sum_{i=1}^Ki^{x}\bigr)^2}. 
\end{align*}
Ignoring the factor $K(K+1)$ and taking the derivative with
respect to $x$ yields
\begin{align}\label{eq:thatderiv}
\frac{\sum_{i=1}^K\sum_{j=1}^K i^xj^{2x-1}(\log(j)-\log(i))}
{\bigl(\sum_{i=1}^Ki^{x}\bigr)^3}.
\end{align}
The numerator in~\eqref{eq:thatderiv} is of the form $\sum_{\ell=1}^K\eta_\ell\log(\ell)$
where $\sum_{\ell=1}^K\eta_\ell=0$.
That quantity $\eta_\ell$ is $\sum_{i=1}^K(i^x\ell^{2x-1}-\ell^xi^{2x-1})$ 
which is a decreasing function of $\ell$ because $x<1/2$.
Next,  because $\log(\ell)$ is non-negative and increasing in $\ell$
we find that the numerator in~\eqref{eq:thatderiv} is negative and 
hence, so is the whole expression.
Therefore $\rho_K(x\giv 1)$ is
a decreasing function of $x$ making $\rho_K(x\giv 1)>\rho_K(1/2\giv1)$
for $x<1/2$. This establishes~\eqref{eq:halfisminimax} for $x<1/2$ and
the case of $x>1/2$ is similar and slightly easier.
\end{proof}

\section{Robustness}\label{sec:robust}

It is not reasonable to suppose that even the more general working model
$\var(\hat\mu_k)\propto k^{-y}$
would hold exactly.  Using the square root rule
$$
\var(\hat\mu) = \sum_{k=1}^k k\var(\hat\mu_k)\Bigm/\Biggl(\,\sum_{k=1}^Kk^{1/2}\Biggr)^2
$$
is a continuous function of $\var(\hat\mu_k)$.
Then small perturbations from a power law make a small difference.
Also a positive function of $k$ that decreases mildly and steadily and convexly as $k$ increases can
be expected to be highly correlated with $k^{-y}$ for some $y$.
The inefficiency of the square root rule in this more general setting is
\begin{align}\label{eq:genrho}
\rho = \sum_{k=1}^K k\var(\hat\mu_k)\sum_{k=1}^K\var(\hat\mu_k)^{-2}\Bigm/\biggl(\,\sum_{k=1}^Kk^{1/2}\biggr)^2.
\end{align}

One plausible alternative to the power law is that $\var(\hat\mu_k)$ might make
steady progress and then plateau.  
An example of that would be $\var(\hat\mu_k) = \min(k,6)^{-1}$
for $k=1,\dots,10$.  We can compute $\rho$ of~\eqref{eq:genrho}
directly in this case and find that using $\omega_k\propto k^{1/2}$
would increase variance by a factor of just under $1.04$ compared to using the optimal weights.
If we suppose that there are $k_1\in\{1,\dots,100\}$ stages with variance decreasing
proportionally to $k^{-1}$ followed by a plateau of $k_2\in\{1,\dots,100\}$ steps with variance proportional
to $(k_1+1)^{-1}$ then the square root rule never raises variance by more than $1.121$ over the optimal combination.

A more general form of mild improvement
has $\sigma^2_k\ge1/k$ to rule out unreasonably good performance,
$\sigma^2_{k+1}\le\sigma^2_{k}$ to model improvement,
and $\sigma^2_k-\sigma^2_{k+1}\ge \sigma^2_{k+1}-\sigma^2_{k+2}$
to model diminishing returns.
Taking $\sigma^2_1=1$ and a fixed $K$, these conditions yield
a convex set of vectors $(\sigma_1^2,\dots,\sigma^2_K)$.
It would be interesting to know the largest inefficiency $\rho$
over that set of possibilities but this $\rho$ is not a concave function
of the $\sigma^2_k$ values and finding the exact worst case is outside the scope of this article.

We can sample the space of possible values by recursively taking $\sigma^2_1=1$,
$\sigma^2_2\sim U[1/2,\sigma^2_1]$, and for $2\le i\le K$,
$$
\sigma^2_{k}\sim  U\Bigl[ \min(1/k,2\sigma^2_{k-1}-\sigma^2_{k-2}),\, \sigma^2_{k-1}\Bigr].
$$
In $10^6$ simulations with $K=5$ the inefficiency was worse than $9/8$ only 18 times
and the worst value was about $1.134$.
In $10^6$ simulations with $K=10$ the inefficiency was worse than $9/8$ 1904 times
and the worst value was about $1.297$.
In $10^6$ simulations with $K=20$ the inefficiency was worse than $9/8$ 101 times
and the worst value was about $1.263$.
There is not much to lose in weighting proportionally to $k^{1/2}$ in this diminshing
returns setup.

If the progress is not steady and convex, then the square root rule can be more inefficient
than described above.
If the variance remains flat for the first $k_1$ steps and then decreases sharply,
the square root rule is less efficient.  For instance if the first $3$ iterations have unit
variance and the next $10$ have variance $0.01$, then the square root rule raises variance
by about $6.37$ fold over the optimal weights.
\cite{mixbeta} advocate putting zero weight on the
first few iterations if such an initial transient is suspected.

\section{Generalization}\label{sec:gene}

We can generalize some aspects of the square root rule
to other power laws.
Suppose that the true rate parameter $y$ is known to satisfy
$L\le y\le U$ for $0\le L\le U<\infty$.
We can then work with $\omega_k\propto k^x$ for
$x=M\equiv (L+U)/2$.
First we recall the definition of $\rho_K$ from~\eqref{eq:ineff}:
\begin{align*}
\rho_K(x\giv y) 
=
\frac{\bigl(\sum_{i=1}^Ki^{2x-y}\bigr)\bigl(\sum_{i=1}^Ki^{y}\bigr)}
{\bigl(\sum_{i=1}^Ki^{x}\bigr)^2}. 
\end{align*}
Then $2M-U=L$ and 
\begin{align}\label{eq:ineffgen}
\rho_K(M\giv U) = \frac{\bigl(\sum_{i=1}^Ki^{L}\bigr)\bigl(\sum_{i=1}^Ki^{U}\bigr)}
{\bigl(\sum_{i=1}^Ki^{M}\bigr)^2}. 
\end{align}

\begin{lemma}\label{lem:supoverygen}
For $\rho_K(x\giv y)$ given by equation~\eqref{eq:ineff},
with $K\ge1$ and $0\le L\le x\le U<\infty$, 
$$\sup_{L\le y\le U} \rho_K(x\giv y)=
\begin{cases}
\rho_K(x\giv U), & x\le M\\
\rho_K(x\giv L), & x\ge M,
\end{cases}
$$
for $M=(L+U)/2$. 
\end{lemma}
\begin{proof}
The proof is the same as for Lemma~\ref{lem:supovery}
because $\rho_K(x\giv y)$ is convex in $y$ for any $K\ge1$
and any $x$ and $\rho_K(x\giv 2x-y)=\rho_K(x\giv y)$. 
\end{proof}

The inequalities in Lemma~\ref{lem:monok} are rather delicate 
and we have not extended them to the more general setting. 
Equation~\eqref{eq:ineffgen} is easy to evaluate for
integer values of $L$, $M$ and $U$. For more general
values, some sharper tools than the integral bounds in this
paper are given by~\cite{burr:talb:1984} who make a detailed
study of sums of powers of the first $K$ natural numbers.
We do see numerically that $\rho_K(M\giv U)$ is nondecreasing
with $K$ in every instance we have inspected.  
Using Proposition~\ref{prop:ibounds}, 
we can easily find the asymptotic inefficiency
$$
\lim_{K\to\infty} \rho_K(M\giv U)  = 
\lim_{K\to\infty} 
\frac{K^L}{L+1}\frac{K^U}{U+1}\Bigm/\Bigl(\frac{K^M}{M+1}\Bigr)^{2}
=\frac{(M+1)^2}{(L+1)(U+1)}.
$$

In cases with $L=0$ and hence $M=U/2$ we get
$(U/2+1)^2/(U+1)$. For instance, an upper bound at the rate $y=U=2$
corresponding roughly to asymptotic 
accuracy of scrambled net integration \citep{smoovar} leads
to $x=M=1$ and an asymptotic inefficiency of at most 
$$
\frac{(1+1)^2}{(0+1)(2+1)}=\frac43.
$$

Some references in Section~\ref{sec:plausible}
describe  problems where the adaptive importance
sampling variance converges exponentially to zero.
It is reasonable to expect that 
each estimate $\hat\mu_k$ will require a large number $n$
of observations to fit a complicated integrand
and that $K$ will then be not too large.

Suppose that $\var(\hat\mu_k)=\tau^2\exp(-yk)$ for some $y>0$.
Then the desired combination is 
$$
\hat\mu(y) = \frac{ \sum_{i=1}^K\exp(iy)\hat\mu_i}{\sum_{i=1}^K\exp(iy)}. 
$$
Not knowing $y$ we use 
$$
\hat\mu(x) = \frac{ \sum_{i=1}^K\exp(ix)\hat\mu_i}{\sum_{i=1}^K\exp(ix)},
$$
for some $x>0$.
The inefficiency of our estimate is now
\begin{align*}
\gamma_K(x\giv y) 
&
\equiv
\frac{ \sum_{i=1}^K\exp((2x-y)i)  \sum_{i=1}^K\exp(yi) }
{\bigl( \sum_{i=1}^K\exp(xi)\bigr)^2}\\[.5ex]
&= 
\Bigl( \frac{e^{2x-y}(e^{K(2x-y)}-1)}{e^{2x-y}-1}\Bigr) 
\Bigl( \frac{e^y(e^{Ky}-1)}{e^y-1}\Bigr) 
\Bigm/
\Bigl( \frac{e^x(e^{Kx}-1)}{e^x-1}\Bigr)^2.
\end{align*}
If $x=y/2$ then the first factor in the numerator is $K$.
It can be disastrously inefficient to use $x\ll y$.
Some safety is obtained in the $x\to\infty$ limit where
$\hat\mu =\mu_K$.  In that limit we consider
using the final and presumably best estimate to
be a sensible default. Then
\begin{align*}
\lim_{x\to\infty}\gamma_K(x\giv y)  =  \frac{e^y(1-e^{-Ky})}{e^y-1} \le\frac{e^y}{e^y-1}.
\end{align*}

We have no direct experience with problems having geometric convergence.
For sake of illustration,
if the variance is halving at each iteration, then $y=\log(2)$
and then taking $x\to\infty$ is inefficient by at most a factor of $2$.
Repeated ten-fold variance reductions correspond to $y = \log(10)$
and a limiting inefficiency of at most $10/9\doteq 1.11$.
The greatest inefficiency from using only the final iteration arises in
the limit $y\to0$ where the factor is $K$.  In this setting, the user is not
getting a meaningful exponential convergence and even there the loss factor
is at most $K$ and, as remarked above, that $K$ is not likely to be large.

We have assumed that $n$ is the same at every iteration.
It is possible to take a varying number $n_k$ of observations at iteration $k$.
\cite{zhan:1996} obtains his convergence rates for $K=2$
assuming  that $n_1$ and $n_2$ are chosen in an asymptotic
regime where $n_1/(n_1+n_2)$ approaches a limit
in the interior of the interval $(0,1)$.
Somewhat unequal values of $n_1$ and $n_2$ optimized the lead constant.

If $\var(\hat\mu_k) = \sigma_k^2/n_k$ with $\sigma^2_k\propto k^{-y}$
for $0\le y\le 1$ and $n_k\propto k^{\alpha}$ then $\var(\hat\mu_k)\propto k^{-y-\alpha}$.
We might then be able to use the more general bounds above with
$L=\alpha$ and  $U=1+\alpha$.
Conceivably, knowing $\sigma^2_k$ would let one tune $n_k$. However
changes to $n_k$ are likely to change $\sigma^2_k$, so that problem is circular
and solving that is outside the scope of this article.

\section{Conclusions}\label{sec:conc}

Some adaptive importance samplers can have geometric convergence
but many other settings will not converge that quickly. 
A simple rule weighing the estimate at stage $k$ proportionally
to $k^{1/2}$ gives an unbiased and nearly efficient estimate 
under a wide range of conditions and it also lets the
user avoid trying to model a random linear combination of $\hat\mu_k$.
An interesting extension would be to specify a well-motivated convex
set of values for $(1,\sigma_2^2/\sigma_1^2,\dots,\sigma^2_K/\sigma^2_1)$
and find a set of weights $\omega_k\ge0$ with $\sum_k\omega_k=1$
that minimizes the worst case inefficiency
$$
\max 
\sum_{k=1}^K\omega_k^2\var(\hat\mu_k)\sum_{k=1}^K\var(\hat\mu_k)^{-2}$$
where the maximum is taken over variances with ratios in the aforementioned
convex set.

\section*{Acknowledgments}
This work was supported by the US NSF
under grants DMS-1521145 and IIS-1837931.
We thank two referees and an associate editor for some
extremely valuable comments.

\bibliographystyle{apalike}
\bibliography{rare}

\begin{thebibliography}{}

\bibitem[Au and Beck, 1999]{au:beck:1999}
Au, S.~K. and Beck, J.~L. (1999).
\newblock A new adaptive importance sampling scheme for reliability
  calculations.
\newblock {\em Structural safety}, 21(2):135--158.

\bibitem[Blanchet et~al., 2007]{blan:glyn:liu:2007}
Blanchet, J., Glynn, P., and Liu, J.~C. (2007).
\newblock Fluid heuristics, {Lyapunov} bounds and efficient importance sampling
  for a heavy-tailed {G/G/1} queue.
\newblock {\em Queueing Systems}, 57(2-3):99--113.

\bibitem[Burrows and Talbot, 1984]{burr:talb:1984}
Burrows, B.~L. and Talbot, R.~F. (1984).
\newblock Sums of powers of integers.
\newblock {\em The American Mathematical Monthly}, 91(7):394--403.

\bibitem[Chatterjee and Diaconis, 2018]{chat:diac:2018}
Chatterjee, S. and Diaconis, P. (2018).
\newblock The sample size required in importance sampling.
\newblock {\em The Annals of Applied Probability}, 28(2):1099--1135.

\bibitem[Cornuet et~al., 2012]{corn:mari:mira:robe:2012}
Cornuet, J., Marin, J.-M., Mira, A., and Robert, C.~P. (2012).
\newblock Adaptive multiple importance sampling.
\newblock {\em Scandinavian Journal of Statistics}, 39(4):798--812.

\bibitem[De~Boer et~al., 2005]{debo:kroe:mann:rubi:2005}
De~Boer, P.-T., Kroese, D.~P., Mannor, S., and Rubinstein, R.~Y. (2005).
\newblock A tutorial on the cross-entropy method.
\newblock {\em Annals of operations research}, 134(1):19--67.

\bibitem[Elvira and Santamaria, 2019]{elvi:sant:2019}
Elvira, V. and Santamaria, I. (2019).
\newblock Multiple importance sampling for efficient symbol error rate
  estimation.
\newblock {\em IEEE Signal Processing Letters}.

\bibitem[Evans and Swartz, 1995]{evan:swar:1995}
Evans, M.~J. and Swartz, T. (1995).
\newblock Methods for approximating integrals in statistics with special
  emphasis on {Bayesian} integration problems.
\newblock {\em Statistical Science}, 10(3):254--272.

\bibitem[Glasserman, 2004]{glas:2004}
Glasserman, P.~G. (2004).
\newblock {\em {Monte Carlo} methods in financial engineering}.
\newblock Springer, New York.

\bibitem[Hesterberg, 1988]{hest:1988}
Hesterberg, T.~C. (1988).
\newblock {\em Advances in importance sampling}.
\newblock PhD thesis, Stanford University.

\bibitem[Kahn and Marshall, 1953]{kahn:mars:1953}
Kahn, H. and Marshall, A. (1953).
\newblock Methods of reducing sample size in {Monte Carlo} computations.
\newblock {\em Journal of the Operations Research Society of America},
  1(5):263--278.

\bibitem[Kollman et~al., 1999]{koll:bagg:cox:pica:1999}
Kollman, C., Baggerly, K., Cox, D., and Picard, R. (1999).
\newblock Adaptive importance sampling on discrete {Markov} chains.
\newblock {\em Annals of Applied Probability}, pages 391--412.

\bibitem[Kong and Spanier, 2011]{kong:span:2011}
Kong, R. and Spanier, J. (2011).
\newblock Geometric convergence of adaptive {Monte Carlo} algorithms for
  radiative transport problems based on importance sampling methods.
\newblock {\em Nuclear Science and Engineering}, 168(3):197--225.

\bibitem[Lai and Siegmund, 1977]{lai:sieg:1977}
Lai, T.~L. and Siegmund, D. (1977).
\newblock A nonlinear renewal theory with applications to sequential analysis
  {I}.
\newblock {\em The Annals of Statistics}, 5(5):946--954.

\bibitem[L'Ecuyer et~al., 2010]{lecu:blan:tuff:glyn:2010}
L'Ecuyer, P., Blanchet, J.~H., Tuffin, B., and Glynn, P.~W. (2010).
\newblock Asymptotic robustness of estimators in rare-event simulation.
\newblock {\em ACM Transactions on Modeling and Computer Simulation (TOMACS)},
  20(1):6.

\bibitem[L'Ecuyer and Lemieux, 2000]{lecu:lemi:2000}
L'Ecuyer, P. and Lemieux, C. (2000).
\newblock Variance reduction via lattice rules.
\newblock {\em Management Science}, 46(9):1214--1235.

\bibitem[L'Ecuyer et~al., 2009]{lecu:mand:tuff:2009}
L'Ecuyer, P., Mandjes, M., and Tuffin, B. (2009).
\newblock Importance sampling and rare event simulation.
\newblock In Rubino, G. and Tuffin, B., editors, {\em Rare event simulation
  using Monte Carlo methods}, pages 17--38. John Wiley \& Sons, Chichester, UK.

\bibitem[L'Ecuyer and Tuffin, 2008]{lecu:tuff:2008}
L'Ecuyer, P. and Tuffin, B. (2008).
\newblock Approximate zero-variance simulation.
\newblock In {\em Proceedings of the 40th Conference on Winter Simulation},
  pages 170--181. Winter Simulation Conference.

\bibitem[Miller, 1986]{mill:1986}
Miller, R.~G. (1986).
\newblock {\em Beyond {ANOVA}: Basics of Applied Statistics}.
\newblock Wiley, New York.

\bibitem[Owen, 1997]{smoovar}
Owen, A.~B. (1997).
\newblock Scrambled net variance for integrals of smooth functions.
\newblock {\em Annals of Statistics}, 25(4):1541--1562.

\bibitem[Owen et~al., 2019]{aloe}
Owen, A.~B., Maximov, Y., and Chertkov, M. (2019).
\newblock Importance sampling the union of rare events with an application to
  power systems analysis.
\newblock {\em Electronic Journal of Statistics}, 13(1):231--254.

\bibitem[Owen and Zhou, 1999]{mixbeta}
Owen, A.~B. and Zhou, Y. (1999).
\newblock Adaptive importance sampling by mixtures of products of beta
  distributions.
\newblock Technical report, Stanford University.

\bibitem[Pharr et~al., 2016]{phar:jako:hump:2016}
Pharr, M., Jakob, W., and Humphreys, G. (2016).
\newblock {\em Physically based rendering: From theory to implementation}.
\newblock Morgan Kaufmann, Cambridge, MA, third edition.

\bibitem[Ryu and Boyd, 2014]{ryu:boyd:2014}
Ryu, E.~K. and Boyd, S.~P. (2014).
\newblock Adaptive importance sampling via stochastic convex programming.
\newblock Technical report, arXiv:1412.4845.

\bibitem[Veach and Guibas, 1995]{veac:guib:1995}
Veach, E. and Guibas, L. (1995).
\newblock Optimally combining sampling techniques for {Monte Carlo} rendering.
\newblock In {\em {SIGGRAPH} '95 Conference Proceedings}, pages 419--428.
  Addison-Wesley.

\bibitem[West, 1993]{west:1993}
West, M. (1993).
\newblock Approximating posterior distributions by mixtures.
\newblock {\em Journal of the Royal Statistical Society: Series B},
  55(2):409--422.

\bibitem[Williams, 1991]{will:1991}
Williams, D. (1991).
\newblock {\em Probability with martingales}.
\newblock Cambridge University Press, Cambridge.

\bibitem[Zhang, 1996]{zhan:1996}
Zhang, P. (1996).
\newblock Nonparametric importance sampling.
\newblock {\em Journal of the American Statistical Association},
  91(435):1245--1253.

\end{thebibliography}
\end{document}